\documentclass[12pt]{amsart}
\usepackage[foot]{amsaddr}
\usepackage[T1]{fontenc}
\usepackage{amssymb,amscd,amsthm,verbatim,amsmath,color,fancyhdr,mathrsfs,mathtools, bbm,bm}
\usepackage{algorithm}
\usepackage[noend]{algpseudocode}
\usepackage{graphicx}
\usepackage{turnstile}
\usepackage{float}
\usepackage{array,longtable}
\usepackage{listings,xcolor}
\definecolor{codegreen}{rgb}{0,0.6,0}
\definecolor{codegray}{rgb}{0.5,0.5,0.5}
\definecolor{codepurple}{rgb}{0.58,0,0.82}
\definecolor{backcolour}{rgb}{0.96,0.96,0.925}

\lstdefinestyle{mystyle}{language=Mathematica,morekeywords={Positive,Min,Permute,ArrayReshape,Tally,GroupElements,GroupOrder,SymmetricGroup,DeleteDuplicates,FirstPosition},
    backgroundcolor=\color{backcolour},   
    commentstyle=\color{codegreen},
    keywordstyle=\color{magenta},
    numberstyle=\tiny\color{codegray},
    stringstyle=\color{codepurple},
    basicstyle=\ttfamily\footnotesize,
    breakatwhitespace=false,         
    breaklines=true,                 
    captionpos=b,                    
    keepspaces=true,                 
    numbers=left,                    
    numbersep=5pt,                  
    showspaces=false,                
    showstringspaces=false,
    showtabs=false,                  
    tabsize=2
}

\newcolumntype{C}{>{$}c<{$}}

\newtheorem{innercustomthm}{Theorem}
\newenvironment{mythm}[1]
  {\renewcommand\theinnercustomthm{#1}\innercustomthm}
  {\endinnercustomthm}

\numberwithin{equation}{section}

\usepackage[letterpaper, left=2.5cm, right=2.5cm, top=2.5cm,bottom=2.5cm,dvips]{geometry}

\usepackage{style} 

\newtheorem{theorem}{Theorem}[section]
\newtheorem{proposition}[theorem]{Proposition}
\newtheorem{lemma}[theorem]{Lemma}
\newtheorem{remark}[theorem]{Remark}

\newtheorem{example}[theorem]{Example}

\newtheorem{conjecture}[theorem]{Conjecture}
\newtheorem{question}[theorem]{Question}

\DeclarePairedDelimiterX{\bracket}[3]{#1}{#2}{#3}
\newcommand{\round}[1]{\bracket*{(}{)}{#1}}

\providecommand{\newoperator}[3]{\newcommand*{#1}{\mathop{#2}#3}}
\providecommand{\renewoperator}[3]{\renewcommand*{#1}{\mathop{#2}#3}}

\renewoperator{\Re}{\mathrm{Re}}{\nolimits}
\renewoperator{\Im}{\mathrm{Im}}{\nolimits}

\DeclarePairedDelimiterXPP{\nrm}[2]{}{\lVert}{\rVert}{\ensuremath{_{#1}}}{\ifblank{#2}{\:\cdot\:}{#2}}
\newcommand{\norm}[2]{\nrm*{#1}{#2}}

\DeclarePairedDelimiterXPP\prob[1]{\mathbb{P}}{\lbrace}{\rbrace}{}{#1} 

\DeclarePairedDelimiterXPP\probability[2]{\mathbb{P}_{#1}}{\lbrace}{\rbrace}{}{#2} 

\DeclarePairedDelimiterXPP\expectation[1]{\mathbb{E}}{\lbrack}{\rbrack}{}{#1} 

\DeclarePairedDelimiterXPP\expectationdist[2]{\mathbb{E}_{#1}}{\lbrack}{\rbrack}{}{#2} 

\DeclarePairedDelimiterXPP\variance[1]{\mathrm{Var}}{\lbrack}{\rbrack}{}{#1} 

\DeclarePairedDelimiterXPP\variancedist[2]{\mathrm{Var}_{#1}}{\lbrack}{\rbrack}{}{#2} 

\DeclarePairedDelimiterXPP\covariance[2]{\mathrm{Cov}}{(}{)}{}{#1,\mathopen{}#2} 

\newoperator{\supp}{\mathrm{supp}}{\nolimits}

\makeatletter
\providecommand*{\diff}
{\@ifnextchar^{\DIfF}{\DIfF^{}}}
\def\DIfF^#1{
	\mathop{\mathrm{\mathstrut d}}
	\nolimits^{#1}\gobblespace}
\def\gobblespace{
	\futurelet\diffarg\opspace}
\def\opspace{
	\let\DiffSpace\!
	\ifx\diffarg(
	\let\DiffSpace\relax
	\else
	\ifx\diffarg[
	\let\DiffSpace\relax
	\else
	\ifx\diffarg\{
	\let\DiffSpace\relax
	\fi\fi\fi\DiffSpace}

\providecommand*{\pdiff}
{\@ifnextchar^{\pDIfF}{\pDIfF^{}}}
\def\pDIfF^#1{
	\mathop{\mathrm{\mathstrut \partial}}
	\nolimits^{#1}\gobblespace}
\def\gobblespace{
	\futurelet\diffarg\opspace}
\def\opspace{
	\let\DiffSpace\!
	\ifx\diffarg(
	\let\DiffSpace\relax
	\else
	\ifx\diffarg[
	\let\DiffSpace\relax
	\else
	\ifx\diffarg\{
	\let\DiffSpace\relax
	\fi\fi\fi\DiffSpace}

\makeatother

\DeclarePairedDelimiterX\Set[1]\{\}{
	
	#1
}

\newcommand{\R}{\mathbb{R}}

\newcommand{\N}{\mathbb{N}}

\newcommand{\Dcal}{\mathcal{D}}
\newcommand{\Ecal}{\mathcal{E}}
\newcommand{\Fcal}{\mathcal{F}}
\newcommand{\Gcal}{\mathcal{G}}

\newcommand{\Pcal}{\mathcal{P}}

\newcommand{\Tcal}{\mathcal{T}}

\newcommand{\Wcal}{\mathcal{W}}

\newcommand{\eq}{\begin{equation}}
\newcommand{\en}{\end{equation}}

\usepackage{amsrefs} 

\title[Marcus--Ree inequality and Erd\H{o}s matrices]{A note on Erd\H{o}s matrices and Marcus--Ree inequality}
\author[Kushwaha]{Aman Kushwaha}
\address{Aman Kushwaha \\ Department of Mathematics \\ University of Delhi\\ Delhi, India \\ {Email: aman\_7778@maths.du.ac.in }} 
\author[Tripathi]{Raghavendra Tripathi}
\address{ Raghavendra Tripathi\\ Department of Mathematics \\ Division of Science\\ New York University\\ Abu Dhabi, UAE\\ {Email: rt1986@nyu.edu}}

	

\thanks{We thank Ludovick Bouthat for carefully reading the first draft of the manuscript and for bringing the reference~\cite{christensen2019comparative} to our attention.}

\date{\today}

\begin{document}

\begin{abstract}
In 1959, Marcus and Ree proved that any bistochastic matrix $A$ satisfies
\[\Delta_n(A)\coloneqq \max_{\sigma\in S_n}\sum_{i=1}^{n}A(i, \sigma(i))-\sum_{i, j=1}^n A(i, j)^2 \geq 0\;.\]
Erd\H{o}s asked to characterize the bistochastic matrices satisfying $\Delta_n(A)=0$. This problem remains largely open, and very recently, a complete list of such matrices was obtained in dimension $n=3$ by Bouthat, Mashreghi, and Morneau-Gu\'erin. Soon after, Tripathi proved that there were only finitely many such matrices in any dimension $n$. In this paper, we continue the investigation initiated in these two works. We characterize all $4\times 4$ bistochastic matrices satisfying $\Delta_4(A)=0$. Furthermore, we show that for $n\geq 3$, $\Delta_n(A)=\alpha$ has uncountably many solutions when $\alpha\in (0, (n-1)/4)$. This answers a question raised in [Tripathi, R., \emph{Some observations on Erd\H{o}s matrices}, Linear Algebra and Its Applications 708 (2025)]. We also extend the Marcus--Ree inequality to infinite bistochastic arrays and bistochastic kernels. Our investigation into $4\times 4$ Erd\H{o}s matrices also leads to several intriguing questions of independent interest. We propose several questions and conjectures and present numerical evidence for them.  
\end{abstract}

\maketitle
\section{Introduction}~\label{sec:Intro}
An $n\times n$ matrix $A$ is said to be \emph{bistochastic} if it has nonnegative entries and the entries in each row and each column of $A$ sum to $1$, that is, $A(i, j)\geq 0$ for all $i, j\in [n]:=\{1, \ldots, n\}$ and
\[ \sum_{j=1}^{n}A(i, j) = \sum_{j=1}^{n}A(j, i) =1, \quad \forall i\in [n]\;.\]
Let $\Omega_n$ be the set of all $n\times n$ bistochastic matrices, also called the Birkhoff polytope. Bistochastic matrices have been intensely investigated for a long time because of their ubiquity. It is easy to see that $\Omega_n$ is a compact, convex set. The famous Birkhoff--von Neumann theorem characterizes the set of extreme points of $\Omega_n$, which is precisely the set of all $n\times n$ permutation matrices, $\mathfrak{S}_n$. It is easy to check that the permanent of a bistochastic matrix can be at most one. In 1926, van der Waerden conjectured that the permanent of an $n\times n$ bistochastic matrix $A$ is at least $n!/n^n$, and it is achieved by the unique bistochastic matrix $J_n\in \Omega_n$ all whose entries are $1/n$. This conjecture was finally proved in 1981 by Egorychev and Falikman. We refer the reader to~\cite{gurvits2007van} for a more recent and elegant proof that highlights the deep connections of this problem with real stable polynomials. 

In 1959, Marcus and Ree~\cite{marcus1959diagonals}, as a first step towards the van der Waerden conjecture, proved that for any $A\in \Omega_n$, we have
\begin{equation}
\label{eqn:MarcusRee}
A\mapsto \Delta_{n}(A) \coloneqq \max_{\sigma\in S_n} \sum_{i=1}^{n}A(i, \sigma(i))-\sum_{i, j=1}^{n}A(i, j)^2\geq 0\;,
\end{equation}
where $S_n$ is the set of all permutations of the set $[n]=\{1, \ldots, n\}$. We refer the reader to the original paper of Marcus and Ree~\cite{marcus1959diagonals} for the connection of this inequality with van der Waerden's conjecture. In the inequality above, the second term is readily recognized as the square of the Frobenius norm of $A$, namely, $\|A\|_{\operatorname{F}}^2 = \sum_{i, j=1}^{n}A(i, j)^2$. The first term is often referred to as the maximal trace, which we shall denote by maxtrace from here on:
\[\mathrm{maxtrace}(A)\coloneqq \max_{\sigma\in S_n} \sum_{i=1}^{n}A(i, \sigma(i))\;.\] We should remark that the $\mathrm{maxtrace}(A)$ can be interpreted as a `tropicalization' of the permanent. More precisely, consider the $\max$-algebra $(\mathbb{R}, \oplus, \cdot)$ where 
\[ x\cdot y = x+y, \qquad \text{and} \quad x\oplus y = \max\{x, y\}\;.\]
The permanent of a matrix $A$ in this algebra will be 
\[ \dot{\operatorname{perm}}(A) := \max_{\sigma\in S_n}\sum_{i=1}^{n}A(i, \sigma(i)) =\mathrm{maxtrace}(A)\;. \]
We also refer the reader to~\cites{bapat1995permanents,butkovivc2003max,burkard2003max} for this interpretation of maximal trace as the permanent in max-algebra and related results. Maximal trace has also been extensively studied in the context of assignment problems. 

Seeing the inequality~\eqref{eqn:MarcusRee}, Erd\H{o}s asked to characterize the matrices $A\in \Omega_n$ for which $\Delta_n(A)=0$. Following~\cite{tripathi2025some}, we refer to such a matrix as an Erd\H{o}s matrix. Notice that both $\operatorname{maxtrace}(\cdot)$ and the Frobenius norm $\norm{F}{\cdot}$ are invariant under pre/post-multiplication by permutation matrices. It follows that $\Delta_n(A)=\Delta_n(PAQ)$ for any permutation matrices $P, Q$. In particular, $A$ is an Erd\H{o}s matrix if and only if $PAQ$ is an Erd\H{o}s matrix for every permutation matrix $P, Q$. We say that bistochastic matrices $A$ and $B$ are \emph{equivalent}, $A\sim B$, if $B=PAQ$ for some permutation matrices $P, Q$. In this light, it is natural to characterize Erd\H{o}s matrices in $\Omega_n$ only up to the equivalence. An easy computation shows that 
\begin{align*}
    I_2 = \begin{pmatrix}
        1 & 0 \\ 0 & 1
    \end{pmatrix},\;\;  J_2 =\frac{1}{2}\begin{pmatrix}
        1 & 1\\ 1 & 1
    \end{pmatrix},\;\; \text{and}\quad P_2=\begin{pmatrix}
         0 & 1\\ 1 & 0
    \end{pmatrix}
\end{align*}
are the only Erd\H{o}s matrices in $\Omega_{2}$. It is easily verified that there are precisely $2$ Erd\H{o}s matrices (up to equivalence), namely, the identity $I$
and the matrix $J_2$. While Marcus and Ree proved some results about Erd\H{o}s matrices, a complete understanding remains elusive. Even for dimension $n=3$, a complete list of Erd\H{o}s matrices (up to the equivalence) was obtained only recently in 2024~\cite{bouthat2024question}. We give the complete list of $3\times 3$ Erd\H{o}s matrices below:
\begin{align*}
I_3 &= \;\;\;\begin{pmatrix}
1 & 0 & 0 \\ 0 & 1& 0\\ 0 & 0 & 1
\end{pmatrix}, & J_3 &= \frac{1}{3}\begin{pmatrix}
1 & 1 & 1\\ 1 & 1 & 1\\ 1& 1 &1
\end{pmatrix}, & I\oplus J_2 &=\frac{1}{2} \begin{pmatrix}
2 & 0 & 0 \\ 0 & 1 & 1\\ 0 & 1 & 1
\end{pmatrix},\\
S &= \frac{1}{4}\begin{pmatrix}
   0 &2 & 2\\
   2 & 1&1\\
   2&1& 1
\end{pmatrix}, & T&= \frac{1}{2}\begin{pmatrix}
    0 & 1 & 1\\
   1 & 0 & 1\\
    1 &1 & 0
\end{pmatrix}, & R &= \frac{1}{5}\begin{pmatrix}
    3 & 0 & 2\\
    0 & 3 & 2\\
    2 & 2 & 1
\end{pmatrix}\;.
\end{align*}
Notice that the matrix $I\oplus J_2$ and $T=\frac{1}{2}(3J_3-I_3)$ naturally generalize to higher dimensions, that is, $\frac{1}{n-1}(nJ_n-I_n)\in \Omega_n$ is an Erd\H{o}s matrix. And, if $A\in \Omega_n, B\in \Omega_m$ are Erd\H{o}s matrices then $A\oplus B\in \Omega_{m+n}$ is an Erd\H{o}s matrix. An important step in the direction of characterizing Erd\H{o}s matrices was taken in~\cite{tripathi2025some}*{Proposition 1.2} where it was shown that every Erd\H{o}s matrix can be written as a convex combination of linearly independent permutation matrices and the convex hull of any collection of linearly independent permutation matrices contains at most one Erd\H{o}s matrix. In particular, \cite{tripathi2025some}*{Theorem 1.3} established that there are only finitely many Erd\H{o}s matrices in $\Omega_n$ for each $n\in \N$ and \cite{tripathi2025some}*{Theorem 1.6} showed that every Erd\H{o}s matrix has rational entries. It was also observed in the discussion following~\cite{tripathi2025some}*{Proposition 4.3} that there are at least $p(n)$ many (non-equivalent) Erd\H{o}s matrices, where $p(n)$ is the number of partitions of $n$. In fact,~\cite{tripathi2025some}*{Proposition 4.3} shows that $\frac{1}{2}(I_n+P)$ is an Erd\H{o}s matrix for any permutation matrix $P$. 
However, many natural questions remain unanswered. We note the following two questions (that remain open):
\begin{enumerate}
    \item Let $\mathfrak{E}_n$ denote the set of all non-equivalent Erd\H{o}s matrices in $\Omega_n$. How does $|\mathfrak{E}_n|$ grow asymptotically?
    \item Is there an efficient way to generate (all) Erd\H{o}s matrices?
    \item It seems that $|\mathfrak{E}_n|\geq n!$. Is there a construction that gives $n!$ many non-equivalent Erd\H{o}s matrices?
\end{enumerate}
We should point out that one can get an upper bound on $|\mathfrak{E}_n|$ using~\cite{tripathi2025some}*{Remark 1.4}, but the upper bound is too loose to be of any use. We know that $|\mathfrak{E}_2|=2$ and $|\mathfrak{E}_3|=6$ from~\cite{bouthat2024question}. Combined with Tripathi's observation~\cite{tripathi2025some}*{Proposition 4.3} that $|\mathfrak{E}_n|\geq p(n)$, it is reasonable to wonder if $|\mathfrak{E}_n|$ grows like $n!$. In this paper, we show that $|\mathfrak{E}_4|=41>4!$.

\subsection*{Contributions of this paper}
It was observed in~\cite{tripathi2025some}*{Proposition 2.1} that $\Delta_n(\Omega_n)=[0, (n-1)/4]$, and that, up to equivalence, there is a unique matrix $M_n=\frac{1}{2}\round{I_n+J_n}\in \Omega_n$ satisfying $\Delta_n(M_n)=(n-1)/4$. More generally,~\cite{tripathi2025some}*{Section 2} introduced $\alpha$-Erd\H{o}s matrices as the matrices $A\in \Omega_n$ satisfying $\Delta_n(A)=\alpha$ for $\alpha\in [0, (n-1/4])$. Note that $0$-Erd\H{o}s matrices correspond precisely to the Erd\H{o}s matrices. An easy calculation shows that there are only finitely many $\alpha$-Erd\H{o}s matrices in $\Omega_{2}$ for any $\alpha\in (0,1/4)$. However, we show that this happens only in dimension $n=2$. More precisely, in Section~\ref{sec:AlphaErdos}, we prove the following. 

\begin{mythm}{~\ref{thm:AlphaErdos_genDim}}
Let $n \ge 3$. If $\alpha\in (0, (n-1)/4)$, then there are (uncountable) infinitely many $n\times n$ symmetric bistochastic matrices $A$ satisfying 
\[\Delta_{n}(A) = \mathrm{maxtrace}(A)-\|A\|_{F}^2 = \alpha\;.\]
\end{mythm}


In Section~\ref{sec:Erdos4}, we compute all $4\times 4$ non-equivalent Erd\H{o}s matrices following the proof technique in~\cite{tripathi2025some}*{Section 3}. 
\begin{mythm}{~\ref{thm:41Erdos}}
    There are exactly $41$ (up to the equivalence) $4 \times 4$ Erd\H{o}s matrices. Further, if we identify matrices up to transposition (i.e., consider $A \sim A^T$), there are $32$ equivalence classes. 
\end{mythm}
We should note that in dimensions $n=2$ and $3$, every Erd\H{o}s matrix is equivalent to a symmetric bistochastic matrix. However, in dimension $n=4$, there are Erd\H{o}s matrices that are not equivalent to any symmetric matrix. This observation was already made in~\cite{bouthat2024question}. It makes sense to identify an Erd\H{o}s matrix $A$ with $A^{T}$ and note that even with this identification, the number of non-equivalent Erd\H{o}s matrices remains $32>4!$. Finding a good lower and upper bound on the number of non-equivalent Erd\H{o}s matrices remains an open problem.

In Section~\ref{sec:Observations_and_Conjectures}, we discuss several problems/conjectures that naturally arise during our investigation of $4\times 4$ Erd\H{o}s matrices. Many of these problems are of independent interest. We also provide numerical results and/or heuristics for these problems/conjectures. We hope that it will provide impetus to further research. 

Finally, in Section~\ref{sec:Generalizations_and_Interpretations} we generalize and prove the Marcus--Ree inequality~\eqref{eqn:MarcusRee} for infinite arrays and kernels. We say that an infinite array $A$ is bistochastic if it has nonnegative entries and each row and column sums to $1$. We prove a version of the Marcus--Ree inequality for such arrays in Section~\ref{sec:MR_infiniteArray}.
\begin{mythm}{~\ref{thm:MR_array}}
Let $A, B$ be bistochastic arrays satisfying 
$\sum_{i, j\in \N}A(i, j)^2,  \sum_{i, j\in \N} B(i, j)^2 <\infty$.
Then, 
\[\langle A, B\rangle \leq \sup_{\sigma\in S_{\N}} \sum_{i=1}^{\infty}A(i, \sigma(i))\;,\]
where $S_{\N}$ is the set of all bijections $\sigma:\N\to \N$ that fix all but finitely many elements in $\N$.
In particular, we also have 
\[\|A\|_{\ell^2}^2\leq \sup_{\sigma\in S_{\N}} \sum_{i=1}^{\infty}A(i, \sigma(i))\;.\]
\end{mythm}
We refer the reader to Example~\ref{exp:BistochasticArray} for an example of a bistochastic array $A$ satisfying the conditions of the above theorem and for which the maximal trace $\sup_{\sigma\in S_{\N}} \sum_{i=1}^{\infty}A(i, \sigma(i))$ is finite.

In Section~\ref{sec:MR_kernel}, we prove a Marcus--Ree inequality for kernels. A \emph{bistochastic kernel} is a bounded function $W:[0, 1]^2\to [0, \infty]$ such that $\int W(x, y)\diff y=\int W(y, x)\diff y=1$ for a.e. $x\in [0, 1]$. For such kernels, we have the following result (slightly paraphrased). We refer the reader to Section~\ref{sec:MR_kernel} for more details.
\begin{mythm}{~\ref{thm:continum}}
Let $W$ be a continuous bistochastic kernel. Then, 
    \[\sup_{\pi\in \Pi}\int W(x, y)\pi(\diff x\diff y) \leq \sup_{T\in \Tcal}\int W(x, T(x))\diff x \;,\]
    where $\Pi$ is the set of all couplings of Lebesgue measure on $[0, 1]$ and $\Tcal$ is the set of all Lebesgue measure preserving maps on $[0, 1]$.
\end{mythm}
We end with an economic interpretation of Marcus--Ree inequality in Section~\ref{sec:EconomicInterpretation}. We hope that these extensions and interpretations will open up several interesting possibilities for future research.

\section{\texorpdfstring{$\alpha$-Erd\H{o}s matrix}{alpha-Erdos matrix}}
\label{sec:AlphaErdos}

The goal of this section is to prove the following theorem.
\begin{theorem}
\label{thm:AlphaErdos_genDim}
Let $n \ge 3$. If $\alpha\in (0, (n-1)/4)$, then there are (uncountable) infinitely many $n\times n$ symmetric bistochastic matrices $A$ satisfying 
\[\Delta_{n}(A) = \mathrm{maxtrace}(A)-\|A\|_{F}^2 = \alpha\;.\]
\end{theorem}

In Proposition~\ref{prop:AlphaErdos3}, we establish a special case of Theorem~\ref{thm:AlphaErdos_genDim} when $n=3$, and we conclude Theorem~\ref{thm:AlphaErdos_genDim} from this special case.  

\begin{proposition}
\label{prop:AlphaErdos3}
    Let $\alpha \in (0,1/2)$. There are (uncountable) infinitely many $3\times 3$ symmetric bistochastic matrices satisfying 
    \[ \|A\|_{F}^2=\mathrm{maxtrace}(A)-\alpha\;.\]
\end{proposition}
The proof of Proposition~\ref{prop:AlphaErdos3} is inspired from~\cite{bouthat2024question} where the authors characterize all $3 \times 3$ Erd\H{o}s matrices. Before we begin the proof, we describe the approach of~\cite{bouthat2024question} for characterizing $3\times 3$ Erd\H{o}s matrices. It is known due to Marcus and Ree that any Erd\H{o}s matrix other than $J_3$ must have a zero-entry. Let $A\in \Omega_3$ be an Erd\H{o}s matrix. Replacing $A$ by $AP$ for some suitable permutation matrix $P$ we can assume that $\mathrm{maxtrace}(A)=\mathrm{Trace}(A)$. Therefore, the authors in~\cite{bouthat2024question} consider a family of bistochastic matrices parametrized by three-parameter $(x, y, z)$
\[A = \begin{pmatrix}
    x & z & * \\ 0 & y & * \\ * & * & *
\end{pmatrix}\;.\]
The authors first characterize the matrices of the above form satisfying $\|A\|_{F}^2=\mathrm{Trace}(A)$. This gives a quadratic equation in $3$ variables, and solving for $z$ in terms of $x$ and $y$, we obtain a two-dimensional feasible region of $(x, y)$ for which $z$ is nonnegative. And, for each point in this feasible region, there are at most two distinct values of $z$. Now, the authors consider the condition $\mathrm{maxtrace}(A)=\mathrm{Trace}(A)$. This determines $5$ linear constraints. Erd\H{o}s matrices correspond to precisely those pairs $(x, y)$ in the feasible region that are in the intersection of these $5$ constraints as well. The authors in~\cite{bouthat2024question} explicitly determine these regions and obtain a finite list of $(x, y)$ that correspond to the Erd\H{o}s matrices.

We now explain how we adapt this approach for our purposes. For a general $\alpha$-Erd\H{o}s matrix, the Marcus--Ree result no longer holds, and thus we need to parametrize a bistochastic matrix with $4$ parameters instead of $3$. As a first attempt, we restrict ourselves to a smaller dimensional sub-manifold of Birkhoff polytope, namely, symmetric bistochastic matrices that we parametrize by 
\[ A = \begin{pmatrix}
    x & z & * \\ 
    z & y & * \\
    * & * & *
\end{pmatrix}\;.\]
Our goal is to find $\alpha$-Erd\H{o}s matrices in this sub-manifold. To this end, for each $\alpha$, we look for a two-dimensional feasible region of $(x, y)$ for which $z$ is nonnegative and such that $A$ satisfies $\|A\|_{F}^2=\mathrm{Trace}(A)-\alpha$. And, we consider the region determined by the constraint $\mathrm{maxtrace}(A)=\mathrm{Trace}(A)$. Somewhat surprisingly, we found that for each $\alpha$, there is an interval $I_{\alpha}$ such that $\{(x, x): x\in I_{\alpha}\}$ is contained in the intersection of the feasible region satisfying both constraints. 

We must point out that unlike~\cite{bouthat2024question}, we do not attempt to characterize all $\alpha$-Erd\H{o}s matrices. Many of our reductions cannot be made without a loss of generality. Therefore, finding all $\alpha$-Erd\H{o}s matrices with this approach seems daunting. However, our goal is to only obtain infinitely many $\alpha$-Erd\H{o}s matrices as this provides a clear contrast between $\alpha\in \{0, (n-1)/4\}$ and $\alpha\in (0, (n-1)/4)$.

\begin{proof}[Proof of Proposition~\ref{prop:AlphaErdos3}]
    For each $\alpha\in (0, 1/2)$, we define an interval $[a(\alpha), b(\alpha)]$ as follows. 
    \[[a(\alpha),b(\alpha)]=
    \begin{cases}
        \left[\frac{5}{12}-\frac{1}{12} \sqrt{1-4 \alpha }\;,\;\frac{2}{3}-\frac{1}{3} \sqrt{1-2 \alpha }\:\right], & \text{if  } 0 \le \alpha \leq \frac{2}{9},\\[7pt]
        \left[\frac{2}{3}-\frac{1}{6} \sqrt{5-10 \alpha }\; , \; \frac{2}{3}-\frac{1}{3} \sqrt{1-2 \alpha }\:\right], & \text{if  } \frac{2}{9} < \alpha <\frac{1}{2}\;.\\
    \end{cases}\]
    It is easy to verify that for each $\alpha$, the interval $[a(\alpha), b(\alpha)]$ has a non-empty interior. Let $\alpha \in (0, 1/2)$ and let $x\in [a(\alpha), b(\alpha)]$, define 
    \[z=z(x)=\dfrac{7-8 x-\sqrt{-36 x^2+48 x-11-10 \alpha}}{10}\;.\]
    We leave the reader the verification that $z$ is a real number such that $z\in [0, 1-x]$ and furthermore $z$ solves the quadratic equation 
    \begin{equation}
    \label{eqn:Quadratic}
         10z^2 + (16x-14)z+ (10x^2-16x+6+\alpha)=0\;.
    \end{equation}
    Define a bistochastic matrix $A=A(x)$ such that
     \[A=  \begin{pmatrix}
     x & z & 1-x-z \\
     z & x & 1-x-z\\
     1-x-z & 1-x-z &  2x+2z-1
     \end{pmatrix}\]
    The proof is now complete if we can show that $A$ satisfies $\|A\|_{F}^2=\mathrm{maxtrace}(A)-\alpha$. This is a straightforward verification. For the convenience of the reader, we outline the steps. We do this in two steps. 
    \subsubsection*{Step I:} We first show that $A$ satisfies $\|A\|_{F}^2=\mathrm{Trace}(A)-\alpha$. To this end, note that by construction, we have
    \begin{align*}
        \|A\|_{F}^2 &= 10z^2 +16 xz -12z +10x^2 -12x +5\;,\\
        \mathrm{Trace}(A)&= 2z+4x-1\;.
    \end{align*}
    Using the fact that $z$ solves~\eqref{eqn:Quadratic}, it is now easily seen that $\|A\|_{F}^2=\mathrm{Trace}(A)-\alpha$.
    \subsubsection*{Step II:}
    We now show that $A$ satisfies $\mathrm{Trace}(A)=\mathrm{maxtrace}(A)$. This involves checking the inequalities $\mathrm{Trace}(A)\geq \mathrm{Trace}(AP)$ where $P$ is a permutation matrix. Since $A$ is symmetric and has only $2$ unique diagonal entries, we need to verify only three inequalities, $\mathrm{Trace}(A)\geq \mathrm{Trace}(AP_{\sigma})$ corresponding to the permutations $\sigma\in \{(12), (23), (123)\}$. We enumerate these inequalities below
    \begin{enumerate}
        \item $x\geq z$,
        \item $5x+4z\geq 3$, and
        \item $2x+z\geq 1$.
    \end{enumerate}
    These inequalities can be easily checked from the formula of $x, z$. We skip the routine calculations.
       
\end{proof}

A priori, it is not clear if one can restrict oneself to a low-dimensional submanifold of $\Omega_n$ where one already has uncountably many solutions to $\Delta_n(A)=\alpha$ so as to adapt the proof of Proposition~\ref{prop:AlphaErdos3} in dimensions $n\geq 4$. At the end of this section, we give a one-parameter family of $\alpha$-Erd\H{o}s matrices in $\Omega_n$. But first, we give a different strategy to prove Theorem~\ref{thm:AlphaErdos_genDim}. It is natural to attempt to construct an $\alpha$-Erd\H{o}s matrix in $\Omega_n$ using the $3\times 3$ matrices. For instance, if $A\in \Omega_3$ is an $\alpha$-Erd\H{o}s matrix, then the matrix $\widetilde{A}\in \Omega_n$ defined as 
\[\widetilde{A}_n \coloneqq \begin{pmatrix}
    A & 0 \\ 0 & I_{n-3}
 \end{pmatrix},\]
 is also an $\alpha$-Erd\H{o}s matrix for the same $\alpha$. This immediately yields uncountably many $\alpha$-Erd\H{o}s matrices in dimension $n\geq 3$ for $\alpha\in (0, 1/2)$. To prove Theorem~\ref{thm:AlphaErdos_genDim}, we need the 
 following result, stated in~\cite{tripathi2025some}*{Proposition 2.1}, where it is attributed to Ottolini and Tripathi.

\begin{proposition}
\label{prop:AlphaMax}
    For each $n\geq 1$, we have
    \[ \max_{A \in \Omega_n}\Delta_n (A)=(n-1)/4 \;.\]
    Moreover, the maximum of $\Delta_n$ on $\Omega_n$ is achieved by the unique (up to the equivalence) matrix $M_n\coloneqq \frac{1}{2}I_n+\frac{1}{2}J_n$.
\end{proposition}

\begin{proof}[Proof of Theorem$~\ref{thm:AlphaErdos_genDim}$]
    Let $n\geq 4$ and let $\alpha\in (0, (n-1)/4)$. As discussed above if $\alpha\in (0, 1/2)$, we can construct uncountably many $\alpha$-Erd\H{o}s matrices
     $\widetilde{X}=X\oplus I_{n-3}$ where $X\in \Omega_{3}$ is an $\alpha$-Erd\H{o}s matrix. Therefore, we assume that $\alpha \in [1/2,(n-1)/4]$. Fix $\alpha_0\in (0, 1/2)$. Let $A\in \Omega_3$ be such that $\Delta_3(A)=\alpha_0$ and let $\widetilde{A}=A\oplus I_{n-3}\in \Omega_n$. Define a curve $t\mapsto \phi_{n,A}(t)\coloneqq t\widetilde{A}+(1-t)M_n$ where $M_n$ is as in~Proposition~\ref{prop:AlphaMax}. Observe that 
   \[\Delta_n(\phi_{n,A}(0))=(n-1)/4, \quad \Delta_n(\phi_{n,A}(1))=\alpha_0\;.\]
   By the intermediate value theorem, there exists $s=s_{A}\in (0, 1)$ such that $\Delta_n(\phi_{n, A}(s_A))=\alpha$. The proof is complete if we show that whenever $A\neq B$, the curves $\phi_{n, A}$ and $\phi_{n, B}$ do not intersect (except at the endpoints). We argue this by contradiction. To this end, let $A\neq B$ be fixed and assume that there exists $t_{A}, t_{B}\in (0, 1)$ such that $\phi_{n, A}(t_{A})=\phi_{n, B}(t_{B})$. If $t_{A}=t_{B}=t\neq 0$, this would mean 
   \[t \widetilde{A}=t\widetilde{B},\]
   which is a contradiction. We now assume that $1>t_{A}> t_{B}>0$ without loss of generality. Note that $\phi_{n, A}(t_{A})=\phi_{n, B}(t_{B})$ implies $t_{A}\widetilde{A}-t_{B}\widetilde{B}=(t_{A}-t_{B})M_n$. But observe that the $(4, 1)$ entry of the matrix $(t_{A}-t_{B})M_n$ is $(t_{A}-t_{B})/2n\neq 0$ while the $(4, 1)$ entry of $t_{A}\widetilde{A}-t_{B}\widetilde{B}$ is $0$ which gives a contradiction.
   
\end{proof}

\begin{proposition}  
    Let $n\geq3$ and $\alpha \in (0,(n-1)/4)$. With each $\alpha$, we associate an interval $I_{n,\alpha}$ (depending on $n, \alpha)$ as follows:
        \[I_{n, \alpha}=\begin{cases}
        \left[\frac{2n-1-\sqrt{1-4 \alpha }}{2 n(n-1)},\;\frac{n+1-\sqrt{(n-1) (n+1-4 \alpha)}}{2 n}\:\right], & \textup{if  }\; 0 < \alpha \leq \frac{n^2-1}{4n^2},\\[10pt]
        \left[\frac{n+1}{2n}-\frac{\sqrt{(n^2-n-1) (n+1-4 \alpha)}}{2n\sqrt{n-1}},\;\frac{n+1-\sqrt{(n-1) (n+1-4 \alpha)}}{2 n}\:\right], & \textup{if  }\; \frac{n^2-1}{4n^2} < \alpha <\frac{n-1}{4}.\\
    \end{cases}\]
    For $x\in I_{n, \alpha}$ we define $z=z(n,\alpha,x)$ such that 
    \[z=\dfrac{2n+1-2(n+1)x}{2 \left(n^2-n-1\right)}\\-\dfrac{\sqrt{-4 n^3 x^2+4 n^3 x+4 n^2 x^2-4 n x-4 \alpha  (n^2-n-1)-3 n^2+n+2}}{2 \left(n^2-n-1\right)\sqrt{n^2-3 n+2}}\;.\]
    Then the $n \times n$ bistochastic matrix $A=A(x)$ such that
     \[A=  \begin{pmatrix}
     x   & z & z & \ldots & z &  *   \\
     z & x & z & \ldots  &z & * \\
     z& z& x & \ldots & z & *\\
     \vdots & \vdots & \vdots & \ddots & \vdots & \vdots\\ 
     z &z &z & \ldots & x&  *\\
     * &*&* & \ldots &*& *
     \end{pmatrix}\]
     is an $\alpha$-Erd\H{o}s matrix for each $x \in I_{n,\alpha}$. 
\end{proposition}

We skip the proof of this proposition which is a verification in the spirit of the proof of Proposition~\ref{prop:AlphaErdos3}. However, we point out that $z$ is actually chosen as a solution of the quadratic $\|A\|_{F}^2=\mathrm{Trace}(A)-\alpha$. So this condition is always satisfied. We then choose a region for $x$ (depending on $\alpha$) so that the discriminant of the quadratic is nonnegative and hence $z$ is a nonnegative real. We make no attempt to optimize for this region. Finding the above intervals is made non-trivial because of further conditions like $z\in [0, 1-x]$ and $\mathrm{maxtrace}(A)=\mathrm{Trace}(A)$. With our choice of intervals, these constraints can be verified by some tedious computation. Our choice of $n\times n$ bistochastic matrix is helpful here. We note that with our choice of matrix $A$, the condition $\mathrm{maxtrace}(A)=\mathrm{Trace}(A)$ amounts to verifying $O(n)$ many inequalities as opposed to $n!$ many. 


\section{Erd\H{o}s matrices in dimension \texorpdfstring{$n=4$}{n=4}}
\label{sec:Erdos4}

In this section, we describe all $4\times 4$ Erd\H{o}s matrices. In the previous section, we outlined the technique used by~\cite{bouthat2024question} for determining all $3 \times 3$ Erd\H{o}s matrices up to equivalence. Recall that two bistochastic matrices $A$ and $B$ are equivalent, $A \sim B$, if there exist permutation matrices $P$ and $Q$ such that $B = PAQ$. Applying this technique directly to the $4 \times 4$ case would require working with an $8$-dimensional submanifold of the Birkhoff polytope. We could adapt this method to find infinitely many $\alpha$-Erd\H{o}s matrices because these matrices are abundant, and we could restrict ourselves to a smaller submanifold. However, Erd\H{o}s matrices are fairly rare, and our objective here is to obtain \emph{all} Erd\H{o}s matrices up to equivalence. Therefore, simplifications such as those made in the previous section are not possible. This makes the approach impractical even for dimension $n=4$. Instead, we follow an algorithm outlined in~\cite{tripathi2025some}*{Section 4}, which reduces the problem of finding Erd\H{o}s matrices to a finite number of computations. While this technique also has limitations and cannot be used for dimensions, $n \geq 5$; with sufficient effort, it allows us to obtain a complete classification of $4 \times 4$ Erd\H{o}s matrices. Our main result is the following:
\begin{theorem}
\label{thm:41Erdos}
There are exactly $41$ (up to the equivalence) $4 \times 4$ Erd\H{o}s matrices. Further, if we identify matrices up to transposition (i.e., consider $A \sim A^T$), there are $32$ equivalence classes. 
\end{theorem}

The list of these $41$ matrices is provided in Appendix~\ref{append:Erdos4x4}. For completeness, we describe the approach used to obtain them. To this end, the following result is crucial for our purposes.

\begin{theorem}[\cite{tripathi2025some}*{Proposition 1.5}]
    Let $\{P_1, \ldots, P_m\}\subseteq \Pcal_n$ be a linearly independent collection of permutation matrices. Let $M\in R^{m\times m}$ be the positive definite matrix such that $M_{i, j}=\langle P_i, P_j\rangle_{\operatorname{F}}$. Set
    \[ {\bf x} = \frac{M^{-1}\mathbbm{1}}{\langle \mathbbm{1}, M^{-1}\mathbbm{1}\rangle}\;,\]
    where $\mathbbm{1}$ is a vector of all $1$s. 
    If $x_i>0$ for all $i\in [m]$, then $\sum_{i=1}^{m}x_iP_i$ is a bistochastic matrix and every Erd\H{o}s matrix is of this form. 
\end{theorem}

This yields the following algorithm for enumerating all $4 \times 4$ Erd\H{o}s matrices, as outlined in ~\cite{tripathi2025some}*{Section 4}.

\begin{algorithm}
\caption{Computing all $4 \times 4$ Erd\H{o}s matrices}
\label{alg:4x4Erdos}
\begin{algorithmic}[1]
    \State Enumerate all linearly independent subsets of $4\times 4$ permutation matrices.
    \State Solve for $M{\bf y}=\mathbbm{1}$.
    \State Discard all subsets for which $y_i\leq 0$ for some $i$.
    \State Set ${\bf x}={\bf y}/\langle {\bf y}, \mathbbm{1}\rangle$ and compute $A = \sum_{i=1}^{m}x_iP_i$.
    \State Verify whether $A$ is an Erd\H{o}s matrix.
\end{algorithmic}
\end{algorithm}

We implemented this algorithm in Mathematica\footnote{The complete code is available at \url{https://github.com/amankoir/erdos-matrices}}. Apart from obtaining the $4\times 4$ Erd\H{o}s matrices, we also observed several curious phenomena that raise interesting mathematical questions of independent interest. Therefore, we summarize our observations in Table~\ref{table:1} and discuss these questions at the end. 

As the dimension of the convex hull of all $n \times n$ permutation matrices, $\Omega_n$, is $(n-1)^2$, it follows from Carath\'eodory's theorem that any linearly independent collection of $n \times n$ permutation matrices can have at most $(n-1)^2+1$ elements. Therefore, it is sufficient to consider subsets containing at most $(4-1)^2+1=10$ permutation matrices of order $4 \times 4$ to obtain all $4 \times 4$ Erd\H{o}s matrices. We now introduce and explain the notations in our table.

\begin{itemize}
    \item For each $k\in \{1, \ldots, 10 \}$, let $\mathcal{P}_k$ denote the family of all subsets of $4 \times 4$ permutation matrices of size $k$ containing identity. We only consider the subsets containing the identity since we are interested in characterizing the Erd\H{o}s matrices only up to equivalence. In particular, $|\mathcal{P}_k|=\binom{23}{k-1}$.
    \item We denote by $\mathcal{I}_k\subseteq \mathcal{P}_k$ the set of all linearly independent subsets of size $k$.
    \item For a subset $X=\{P_1, \ldots, P_k\}\in \mathcal{I}_{k}$, let $M_{X}$ to be the $k\times k$ Gram matrix with $M_{X}(i, j)=\langle P_i, P_j\rangle_{F}$. Note that the definition of $M_{X}$ depends on the choice of the ordering of the elements in the subset $X$. We ignore this by an abuse of notation as the exact order will not be important to us. Let ${\bf y}$ be the unique solution to $M_{X}{\bf y}=\mathbbm{1}$. Let $\mathcal{G}_k$ denote the set of all $X\in \mathcal{I}_{k}$ such that ${\bf y}$ satisfies ${y}_i>0$ for all $i\in [k]$. Note that because of this last condition the choice of ordering in the set $X$ does not matter.  Let ${\bf x}={\bf y}/\langle \mathbbm{1}, {\bf y}\rangle$ and define a bistochastic matrix $A_{X}=\sum_{i=1}^{k}x_iP_i$.
    \item At this stage, we are left with a few thousand matrices that are candidates for being Erd\H{o}s matrices. We then check each candidate to determine whether it satisfies the defining condition. Let $\mathcal{E}_{k}$ denote the set of all those subsets $X\in \Gcal_k$ for which $A_{X}$ is an Erd\H{o}s matrix. 
    \item Many of these subsets may give the same/equivalent Erd\H{o}s matrices. Let $\mathcal{H}_k$ denote the set of non-equivalent Erd\H{o}s matrices $A_{X}$ as $X$ runs in $\Ecal_k$.
\end{itemize}

Finally, $\bigcup_{k=1}^{10} \mathcal{H}_k$ might still have equivalent Erd\H{o}s matrices therefore we check them for equivalence. This yields $\left| \bigcup_{k=1}^{10} \mathcal{H}_k \big/\sim\right| = 41$, which proves the first part of Theorem~\ref{thm:41Erdos}. The second part follows by identifying pairs $(A,B)$ of non-equivalent Erd\H{o}s matrices that satisfy $A =PB^TQ$ for some permutation matrices $P$ and $Q$.

\begin{table}
\centering
\begin{tabular}{|l |l |l |l| l| l|}
\hline
$k$ & $\binom{23}{k-1}$ & $|\mathcal{I}_{k}|$ & $|\mathcal{G}_k|$ & \rule{0pt}{12pt} $|\mathcal{E}_k|$&$|\mathcal{H}_k|$\\[0.2em]
\hline
$1$ &$1$ &$1$ &$1$ &$1$& $1$ \\
$2$ &$23$ &$23$ &$23$&$23$ & $4$\\
$3$ &$253$ &$253$ &$244$&$172$& $6$\\
$4$ &$1771$ &$1768$ &$1400$&$580$& $14$\\
$5$ &$8855$ &$8780$ &$5275$ &$1040$& $11$\\
$6$ &$33649$ &$32736$ &$11652$ &$1086$ &$14$\\
$7$ &$100947$ &$93765$ &$17059$&$1589$& $6$\\
$8$ &$245157$ &$204688$ &$14456$&$416$& $3$\\
$9$ &$490314$ &$320304$ &$5286$&$300$& $1$\\
$10$ &$817190$ &$287180$&$30$& $30$&$1$\\
\hline
\end{tabular}
\caption{Intermediate results from computation of Erd\H{o}s matrices}
\label{table:1}
\end{table}

\begin{remark}
Note that the first step of Algorithm~\ref{alg:4x4Erdos} requires identifying all linearly independent subsets of $n \times n$ permutation matrices, which involves examining all subsets up to size $(n-1)^2+1$. While this is computationally manageable in dimension $n=4$, the number of such subsets becomes prohibitively large in higher dimensions, making this approach infeasible. 
\end{remark}

\subsection{Some observations and questions}
\label{sec:Observations_and_Conjectures}

\subsubsection*{Linear independence of a random subset of permutation matrices}
We refer the reader to Table~\ref{table:1}. This table reveals that a high fraction of the subsets of $4\times 4$ permutation matrices (of size $\leq 10$) are linearly independent. This inspired us to look further into the linear independence of random subsets of $n\times n$ permutation matrices. Of course, any subset of size bigger than $(n-1)^2+1$ is linearly dependent. Therefore, we focus our attention to a random subset of size $d_{n}\coloneqq (n-1)^2+1$. There are $\binom{n!}{d_n}\sim e^{n^3\log n}$ subsets of $n\times n$ permutation matrices of size $d_n$. Therefore, checking the linear (in)dependence of all such subsets is computationally infeasible even in dimension $n=5$. Therefore, we propose to study the following probabilistic variant. What is the probability that a randomly chosen subset of permutation matrices of size $d_n$ is linearly independent? For each $4\leq n\leq 16$, we generate $n^5$ subsets of permutation matrices of size $d_n$ and count how many of these subsets are linearly dependent. 
We refer the reader to Algorithm~\ref{alg:LinearDependence} for the details of the implementation and Table~\ref{table:2} for the results of our simulation. In Table~\ref{table:2}, $d$ denotes the number of subsets (out of $ n^5 $) that are linearly dependent. The empirical probability of a subset being linearly dependent, given by $d/n^5$, is also recorded in Table~\ref{table:2}.  

\begin{algorithm}
\caption{Linear dependence of a random subset of permutations}
\label{alg:LinearDependence}
\begin{algorithmic}
\Require $ n \geq 4 $.
\State Define $ k := (n-1)^2 + 1 $, $ N_{\text{iter}} := n^5 $.
\State Initialize $ d \gets 0 $, $ i \gets 0 $.
\While{$d + i \leq N_{\text{iter}}$}
    \State Generate $n \times n$ permutation matrices $P_1, P_2, \dots, P_k$, chosen uniformly at random.
    \State Vectorize $P_i \in \R^{n\times n}\to P_i \in \R^{n^2}$ for each $i$.
    \State Set $\mathcal{S} = \bigcup_{i=1}^{k}\{P_i\}$ . \Comment{This removes duplicate permutations}
    \If{$|\mathcal{S}|=k$}
        \State Identify $\mathcal{S}=\{P_1,P_2,\ldots,P_k\}$ with $\big[ P_1 \; P_2 \; \dots \; P_k \big]^T \in \mathbb{R}^{k \times n^2} $
        \If{rank $\mathcal{S}^{T}<k$}
            \State $d \gets d + 1$
        \Else
            \State $i \gets i + 1$
        \EndIf
    \EndIf
\EndWhile
\State \textbf{return} $d$
\end{algorithmic}
\end{algorithm}

\begin{table}
\centering
\captionsetup{justification=centering}
\begin{tabular}{|l |l |l |l |} 
&&&\\[-1.2em]
\hline\\[-0.9em]
$n$ & $n^5$ & $d$ &$d/n^5$\\[2pt]
&&&\\[-1.2em] 
\hline \\[-0.9em]
5&3125&2416&$0.77312$\\
6&7776&4459&$0.573431$\\
7&16807&3735&$0.222229$\\
8&32768&2794&$0.0852661$\\
9&59049&2229&$0.0377483$\\
10&100000&1775&$0.01775$\\
11&161051&1363&$8.46315\times10^{-3}$\\ 
12&248832&895&$3.5968\times10^{-3}$\\
13&371293&594&$1.59981\times10^{-3}$\\
14&537824&340&$6.32177\times10^{-4}$\\
15&759375&219&$2.88395\times10^{-4}$\\
16&1048576&122&$1.16348\times10^{-4}$\\  
\hline
\end{tabular}
\caption{Empirical probabilities of random subsets of permutations being linearly dependent}
\label{table:2}
\end{table}

Based on the results of Table~\ref{table:2}, we make the following conjecture.
\begin{conjecture}
\label{conj:LinearIndependence}
    Let $n\geq 4$. Let $\mathcal{S}_n$ be a random subset of $n\times n$ permutation matrices of size $d_n$, then $\mathcal{S}_n$ is linearly independent with high probability. More precisely, 
    \[\mathbb{P}\round{\mathcal{S}_n\text{ is linearly independent}}= 1-o_n(1)\;.\]
\end{conjecture}

This conjecture can be reformulated in terms of the invertibility probability of a positive semi-definite matrix as follows. Let $\mathcal{S}$ be a random subset of size $d_n$ of $n\times n$ permutation matrices. Let $M=M_{\mathcal{S}}$ denote the positive semi-definite matrix 
\[ M(i, j) =\langle P_i, P_j\rangle_{\operatorname{F}}\;, \quad i, j\in [d_n]\;.\]
Conjecture~\ref{conj:LinearIndependence} is equivalent to showing that $\mathbb{P}(\det(M_{\mathcal{S}})\neq 0)=1-o_n(1)$. This is reminiscent of a famous problem in random matrix theory~\cites{tikhomirov2020singularity,tao2007singularity} regarding the invertibility of a random Bernoulli matrix, that is, a random matrix with i.i.d. $0-1$ entries. We should point out that Conjecture~\ref{conj:LinearIndependence} differs from the invertibility of random Bernoulli matrix in that the entries of the matrix $M$ are highly correlated. The invertibility of a random Bernoulli matrix with dependent entries has been recently studied in~\cite{jain2020sharp}. While the setting in~\cite{jain2020sharp} is quite different, we believe that some of the ideas in~\cite{jain2020sharp} may be useful in our setting.



\subsubsection*{On the positivity of the solution of $M{\bf x}=\mathbbm{1}$}
What we have argued above is that most subsets of permutation matrices (of appropriate sizes) seem to be linearly independent. However, the Erd\H{o}s matrices are rare. Looking at Algorithm~\ref{alg:4x4Erdos} and Table~\ref{table:1}, it seems that there are two important sources for the rarity of Erd\H{o}s matrices. First, we notice that $|\Gcal_k|$ is very small compared to $|I_k|$. In other words, for most linearly independent subsets of permutation matrices, the solution to $M{\bf y}=\mathbbm{1}$ has a non-positive coordinate. For a positive semi-definite matrix $A$, the various conditions for the solution of $A{\bf x}=b$ to have positive coordinates have already been investigated in matrix theory~\cite{pena2001class} and economics literature~\cite{christensen2019comparative}. It is natural to wonder if one can prove a quantitative bound on the probability that $M{\bf x}=\mathbbm{1}$ has positive solution when $M$ is a Gram matrix given by $M=\langle P_i, P_j\rangle_{F}$ where $(P_i)_{i\in [m]}$ are i.i.d. uniform permutation matrices.

\subsubsection*{Number of orbits of subsets of permutations}
We finally bring the reader's attention to another curious phenomenon in Table~\ref{table:1}. Looking at the last two columns in Table~\ref{table:1} shows that Erd\H{o}s matrices typically have large equivalence classes. Recall that we say two bistochastic matrices $A$ and $B$ are equivalent, $A\sim B$, if $A=PBQ$ for some permutation matrices $P, Q$. Given a bistochastic matrix $A$, can we understand the size of the equivalence class of $A$? This leads us to propose the following combinatorics problem that may be of independent interest. 

Let $n\in \mathbb{N}$ and let $k\in [n!]$. Let $\mathcal{P}_{n, k}$ be the family of all subsets of size $k$ of the permutations in $S_n$. That is, 
\[\mathcal{P}_{n, k} \coloneqq \{S\subseteq S_n: |S|=k\}\;.\]
Given $S=\{\sigma_1, \ldots, \sigma_k\}, T=\{\tau_1, \ldots, \tau_k\}\in \mathcal{P}_{n, k}$ we declare $S\sim T$ if there exist permutations $\mu, \nu\in S_n$ such that 
\[T= \mu S\nu\coloneqq \{\mu\sigma_1\nu, \ldots, \mu\sigma_k\nu\}\;.\]
It is easy to check that $\sim$ is an equivalence relation on $\mathcal{P}_{n, k}$. 
\begin{question}
\label{ques:Equivalence}
Determine the number of equivalence classes in $\mathcal{P}_{n, k}$ under $\sim$, That is, determine $f_{n,k}\coloneqq |\mathcal{P}_{n, k}/\sim|$. Does $f_{n,k}$ satisfy some recurrence?
\end{question}
It is easy to note that $f_{n,1}=f_{n,n!}=1$ for any $n$. It is also easy to show that $f_{n,2}=p(n)-1$ for all $n \geq 2$. 

\begin{lemma}
\label{lem:fn2}
    For any $n\geq 2$, we have $f_{n,2}=p(n)-1$ where $p(n)$ is the number of partitions of $n$. 
\end{lemma}
\begin{proof}
Let $n\geq 2$ and let $S=\{\sigma_1, \sigma_2\}\in \mathcal{P}_{n, 2}$. Clearly $S\sim \{e, \sigma_1^{-1}\sigma_2\}$. It suffices to classify the subsets of the form $\{e, \sigma\}$ up to the equivalence. To this end, note that $S=\{e, \sigma\}, T=\{e, \tau\}\in \mathcal{P}_{n, 2}$ are equivalent if and only if $\sigma$ and $\tau$ are conjugates. As the number of conjugacy classes in $S_n$ is $p(n)$ and $\sigma$ can not be identity, it follows that $f_{n,2}=|\mathcal{P}_{n, 2}/\sim|=p(n)-1$. 
\end{proof}

Furthermore, it can also be observed that the equivalence relation $\sim$ intertwines the bijection $S\mapsto S_n\setminus S$ between $P_{n,k}$ and $P_{n,n!-k}$. This immediately yields the following lemma. 

\begin{lemma}
\label{lem:fnksymmetric}
    For any $n\in \N$, and $k\in [n!]$ we have $f_{n,k}=f_{n,n!-k}$. 
\end{lemma}

In Table~\ref{tab:nonequivalentsubsets} we give $f_{n,k}$ for some small values of $n$ and $k$. For $n=3$, the numbers $(f_{n,k})_{k=1}^{5}$ can be found in~\cite{tripathi2025some}*{Section 4.1}. 

\begin{table}[h!]
    \centering
    \begin{tabular}{|l|l|l|l|}
      \hline
      $k$ & $n=3$ & $n=4$ & $n=5$  \\  
      \hline
      $1$ & $1$ & $1$ & $1$  \\  
      $2$ & $2$ & $4$ & $6$  \\  
      $3$ & $2$ & $10$ & $37$  \\  
      $4$ & $2$ & $41$ & $715$  \\  
      $5$ & $1$ & $103$ & $13710$  \\  
      $6$ & $1$ & $309$ & $256751$  \\  
      $7$ &  & $691$ & $4140666$  \\  
      $8$ &  & $1458$ & $58402198$  \\  
      $9$ &  & $2448$ & $726296995$  \\  
      $10$ &  & $3703$ & $8060937770$  \\  
      $11$ &  & $4587$ & $80604620206$  \\  
      $12$ &  & $5050$ & $732149722382$  \\  
      \hline
    \end{tabular}
    \caption{Number of non-equivalent subsets of size $k$ in $S_n$}
    \label{tab:nonequivalentsubsets}
\end{table}

It seems that the sequence $(f_{n, k})$ has not appeared before our work. We submitted this sequence on OEIS and it now appears as OEIS sequence A381842~\cite{oeisA381842}. The data for $n=5$ was provided on OEIS by Andrew Howroyd. In this regard, we should also mention that OEIS sequence A362763~\cite{oeisA362763} seems closely related to our sequence $(f_{n, k})$. It would be interesting to determine the order of growth for either of these sequences.
\section{Marcus--Ree inequality for infinite arrays and kernels}
\label{sec:Generalizations_and_Interpretations}
\subsection{Marcus--Ree inequality for infinite bistochastic arrays}
\label{sec:MR_infiniteArray}
A \emph{bistochastic array} is an infinite array $(A(i, j))_{i, j\in \N}$ such that $A(i, j)\geq 0$ for all $i, j\in \N$ and satisfies
\[\sum_{k=1}^{\infty} A(i, k)=1=\sum_{k=1}^{\infty}A(k, i), \quad \forall i\in \N.\]

Observe that any $n \times n$ bistochastic matrix $X$ can be naturally identified with a bistochastic array $K(X)$ where
\[
    K(X)(i, j) = \begin{cases}
                X(i, j), & i,j \in [n],\\
                \delta_{i=j}, & \text{otherwise}\;.
    \end{cases}
\]

In particular, the set of bistochastic arrays is indeed non-empty. We can naturally extend the usual Frobenius inner product on matrices to the bistochastic arrays as follows. Let $A$ and $B$ be two bistochastic arrays; we define 
\[\langle A, B\rangle = \sum_{i, j=1}^{\infty}A(i, j)B(j, i)\;.\]
Since the entries of $A$ and $B$ are nonnegative, the above quantity is always well-defined--albeit it may be equal to $+\infty$. This inner-product naturally defines a norm on bistochastic arrays and we let $\Omega_{\infty, 2}$ denote the set of all bistochastic arrays $A$ with the additional condition that \[\|A\|_{\ell^2}^2\coloneqq \langle A, A\rangle = \sum_{i, j=1}^{\infty} A(i, j)^2<\infty\;.\]

In Theorem~\ref{thm:MR_array} we establish an analog of the Marcus--Ree inequality~\eqref{eqn:MarcusRee} for bistochastic arrays. Before we state our result, we define the maximal trace for a bistochastic array. Let $S_{\mathbb{N}}$ denote the set of all permutations of $\N$, that is, 
\[S_{\mathbb{N}}=\{\sigma:\N\to \N: \sigma \text{ is bijective, and }\sigma(i)=i \text{ for all but finitely many } i\}\;.\]
We define the maximal trace for the bistochastic array $A$ as 
\[\mathrm{maxtrace}(A)\coloneqq \sup_{\sigma\in S_{\N}} \sum_{i=1}^{\infty}A(i, \sigma(i))\;.\]

\begin{theorem}
\label{thm:MR_array}
Let $A, B\in \Omega_{\infty, 2}$. Then, 
\[\langle A, B\rangle \leq \sup_{\sigma\in S_{\N}} \sum_{i=1}^{\infty}A(i, \sigma(i))\;.\]
In particular, we also have 
\[\|A\|_{\ell^2}^2\leq \sup_{\sigma\in S_{\N}} \sum_{i=1}^{\infty}A(i, \sigma(i))\;.\]
\end{theorem}

Note that for any $n\times n$ bistochastic matrix $X$, the  $\ell_2$ norm of $K(X)$ is infinite. Theorem~\ref{thm:MR_array} is interesting only when the maximal trace is finite. We refer the reader to Example~\ref{exp:BistochasticArray} for a non-trivial example of a bistochastic array that has finite maximal trace and finite $\ell_2$ norm. We now state and prove Marcus--Ree inequality for bistochastic arrays. Before we begin the proof, we need another notation. 

Let $X\in \Omega_n$ be an $n\times n$ bistochastic matrix. We will often identify $X$ with an infinite array ${\bf X}$ such that ${\bf X}(i, j)=X(i, j)$ for $i, j\in [n]$ and $0$ otherwise. For any infinite array $A$, we note that 
\[ \langle A, {\bf X}\rangle = \langle A_n, X\rangle_{F}=\sum_{i, j=1}^{n}A(i, j)X(i, j)\;.\]
\begin{proof}[Proof of Theorem$~\ref{thm:MR_array}$]
    The second inequality follows from the first by taking $B=A^{T}$. Therefore, it suffices to prove only the first part.
    If $\sum_{i=1}^{\infty}A(i, \sigma(i))=+\infty$, there is nothing to prove. Therefore, we assume that $\sum_{i=1}^{\infty}A(i, \sigma(i))<+\infty$.
    
    Consider the map 
    \[\Omega_n \ni X\mapsto \Lambda(X) \coloneqq \langle A, {\bf X}\rangle =\sum_{i, j=1}^{n}A(i, j)X(i, j)\;.\]
    Since $\Omega_n$ is convex and $X\mapsto \Lambda(X)$ is linear, it follows that $\Lambda$ attains its maximum at some extreme point of $\Omega_n$. Since the extreme points of $\Omega_n$ are precisely the permutation matrices by Birkhoff-von Neumann theorem~\cite{birkhoff1946three}, it follows that 
    \begin{equation}
    \label{eqn:FiniteReduction}
        \Lambda(X) \leq \sup_{\sigma\in S_n}\sum_{i=1}^{n}A(i, \sigma(i))+\sum_{i=n+1}^{\infty}A(i, i)\leq \sup_{\sigma\in S_{\N}}\sum_{i=1}^{\infty}A(i, \sigma(i))\;,
    \end{equation}
    for any $X\in \Omega_n$ and $n\in \N$.
  
    Now let $B\in \Omega_{\infty, 2}$. Let $B_{n}$ be the $n\times n$ principal submatrix of $B$, that is, $B_n(i, j)=B(i, j)$ for $1\leq i, j\leq n$. Observe that
    \[ \langle A, B\rangle =  \langle A, {\bf B_n}\rangle +  \langle A, B-{\bf B_n}\rangle \leq \langle A, {\bf B_n}\rangle +  \|A\|_{\ell^2}\|B-{\bf B_n}\|_{\ell^2}\;,\]
    where we used the Cauchy-Schwarz inequality for the second term in the last inequality. It is easy to see that $\|B-{\bf B_n}\|_{\ell^2}\to 0$ as $n\to \infty$. For the first term, we wish to appeal to~\eqref{eqn:FiniteReduction}. However, the matrix $B_n$ is not bistochastic, and hence we can not directly appeal to~\eqref{eqn:FiniteReduction}. To fix this issue, we take a bistochastic extension $C_n$ of $B_n$. 
    That is, suppose that there exists an $N\times N$ bistochastic matrix $C_n$ for some $N\geq n$ such that $C_n(i,  j)=B_n(i, j)$ for $1\leq i, j\leq n$. Then,  
    \begin{align*}
        \langle A, {\bf B_n} \rangle = \langle A, {\bf C_n}\rangle + \langle A, {\bf B_n}-{\bf C_n}\rangle \leq \langle A, {\bf C_n}\rangle,
    \end{align*}
    where the last inequality follows from the fact that ${\bf B_n} - {\bf C_n} \leq 0$ by construction. We now appeal to~\eqref{eqn:FiniteReduction} to conclude that 
    \[ \langle A, {\bf C_n}\rangle \leq \sup_{\sigma\in S_{\N}}\sum_{i=1}^{\infty}A(i, \sigma(i))\;. \]
    The  proof is therefore complete by Lemma~\ref{lem:ExtensionLemma} that guarantees the existence of $C_n$ and letting $n\to \infty$.
\end{proof}

\begin{lemma}
\label{lem:ExtensionLemma}
   Let $B$ be $n\times n$ matrix with nonnegative entries. Let 
   \[r_{i}\coloneqq \sum_{k=1}^{n}B(i, k), \quad c_{i}\coloneqq \sum_{k=1}^{n}B(k, i), \quad \forall i\in [n]\;.\]
   If $r_i, c_i \leq 1$ for all $i\in [n]$, then there exists an $N\times N$ bistochastic matrix $D$ for some $N\geq n$ such that 
   \[ D(i, j)=B(i, j), \quad i, j\in [n]\;.\]
\end{lemma}
\begin{proof}
If $B$ is already bistochastic, there is nothing to be done. Assume that $B$ is not bistochastic.
Let $R$ be an $n\times n$ diagonal matrix such that $R(i, i)=1-r_i$ and let $C$ be an $n\times n$ diagonal matrix such that $C(i, i)=1-c_i$.  
Let $X$ be an $n\times n$ matrix such that 
\[ X(i, j)=\frac{c_i\cdot r_j}{\|B\|_{1}}\;,\]
with the convention that $0/0=0$.
It is easy to check that $X$ satisfies 
\begin{equation}
\label{eqn:Xcondition}
    \sum_{k=1}^{n}X(i, k) = c_i, \quad \sum_{k=1}^{n}X(k, i) = r_i,
\end{equation}
for all $i\in [n]$. 
Now consider the $2n\times 2n$ matrix $D$, defined as 
\[ D=\begin{pmatrix} A & R \\
C & X
\end{pmatrix}\;.
\]
It is easy to verify that $D$ is a bistochastic matrix using~\eqref{eqn:Xcondition}. 
\end{proof}

\begin{example}
\label{exp:BistochasticArray}
As promised, we close this section with an example of a bistochastic array with finite maximal trace and we verify that it satisfies the Marcus--Ree inequality. We will define an infinite array $A$ recursively. We set the first row and first column as 
\[A(1,1)=A(1,2)=A(2,1)=\frac{1}{2} \text{ and }  A(1,j)=A(j,1)=0\;,\quad \forall j\geq 3\;.\] 
Given the first $k-1$ rows and columns for some $k\geq 2$, we fill in the remaining entries in the $k$-th row and column with $1/2^k$ so that the $k$-th row and the column sums to $1$. More precisely, we set
\[ A(k, i)=A(i, k)=\frac{1}{2^k}\;, k \leq i \leq m_k\;, \]
where $m_k$ is chosen so that
\begin{equation}
\label{eqn:ConditionM_k}
    \sum_{i=1}^{m_k}A(k, i)= \sum_{i=1}^{k-1} A(k, i)+ \frac{(m_k-k+1)}{2^k}=1\;.
\end{equation}

In particular,
 
 \[A=\left(\begin{array}{ccccccccccc}
\frac{1}{2} & \frac{1}{2} & 0 & 0 & 0 & 0 & 0 & 0 & 0 & \cdots \\[1.2ex]
\frac{1}{2} & \frac{1}{4} & \frac{1}{4} & 0 & 0 & 0 & 0 & 0 & 0 & \cdots\\[0.7ex]
0 & \frac{1}{4} & \frac{1}{8} & \frac{1}{8}& \frac{1}{8}& \frac{1}{8}& \frac{1}{8}& \frac{1}{8}& 0&  \cdots  \\[0.7ex]
0 & 0 &  \frac{1}{8}  & \frac{1}{16} &\frac{1}{16} &\frac{1}{16} &\frac{1}{16} &\frac{1}{16} &\frac{1}{16}  &\cdots\\[0.7ex]
0 & 0 &  \frac{1}{8}  & \frac{1}{16} & \frac{1}{32} & \frac{1}{32} &\frac{1}{32} &\frac{1}{32} &\frac{1}{32} & \cdots\\[0.7ex]
0 & 0 &  \frac{1}{8}  &\frac{1}{16} & \frac{1}{32} & \frac{1}{64} & \frac{1}{64} &\frac{1}{64} &\frac{1}{64}  &  \cdots \\[0.7ex]
0 & 0 &  \frac{1}{8}  &\frac{1}{16} &\frac{1}{32} & \frac{1}{64} & \frac{1}{128} & \frac{1}{128} & \frac{1}{128} &  \cdots\\[0.7ex]
0 & 0 &  \frac{1}{8}  &\frac{1}{16} & \frac{1}{32} &\frac{1}{64} & \frac{1}{128} & \frac{1}{256} & \frac{1}{256} &  \cdots\\[0.7ex]
0 & 0 & 0 & \frac{1}{16} & \frac{1}{32} & \frac{1}{64} & \frac{1}{128} & \frac{1}{256} & \frac{1}{512} &  \cdots  \\[0.7ex]
\vdots &\vdots &\vdots &\vdots &\vdots &\vdots &\vdots &\vdots &\vdots   & \ddots \\[0.7ex]
\end{array}\right)
\]
We now show that 
\[
\|A\|_{\ell^2}^2\leq \sup_{\sigma\in S_{\N}} \sum_{i=1}^{\infty}A(i, \sigma(i))<\infty.
\]
We begin by verifying that $\sup_{\sigma\in S_{\mathbb{N}}} \sum_{i=1}^{\infty}A(i, \sigma(i)) \leq 2$. Note that $\frac{1}{2^k}$ can appear at most twice in the sum $\sum_{i=1}^{\infty}A(i, \sigma(i))$ for any $\sigma\in S_{\N}$. In particular,

\[
\sum_{i=1}^{\infty}A(i, \sigma(i)) \leq \sum_{i=1}^{\infty}\frac{2}{2^i} = 2,
\]
for any $\sigma \in S_{\mathbb{N}}$.

Now, we establish an upper bound for $ \|A\|_{\ell^2}^2$. Note that $\frac{1}{2^k}$ appears in $A$ only in the $k$-th row and $k$-th column. Furthermore, it follows from~\eqref{eqn:ConditionM_k} that $\frac{1}{2^k}$ appears at most $2^{k+1}-1$ times in $A$. 

It follows that 
\begin{align}
    \|A\|_{\ell^2}^2 &= \sum_{\stackrel{i, j\in \N}{\min\{i, j\}\leq 5}} A(i, j)^2 + \sum_{i, j\geq 5} A(i, j)^2 \nonumber\\
    &\leq 3\frac{1}{2^2} +3\frac{1}{4^2}+11\frac{1}{8^2}+27\frac{1}{16^2}+51\frac{1}{32^2} + \sum_{i=6}^{\infty}\frac{(2^{i+1}-1)}{2^{2i}}\nonumber\\
    &<  3\frac{1}{2^2} +4\frac{1}{4^2}+11\frac{1}{8^2}+27\frac{1}{16^2}+51\frac{1}{32^2}\leq 1.328\;.
    \label{eqn:l2bound}
\end{align}

Finally, we verify that $A$ satisfies the Marcus--Ree inequality. To do this, we show that the maximal trace is at least $4/3$. To this end, let $\sigma: \mathbb{N} \to \mathbb{N}$ be given by  
\[
\sigma(2i) = 2i-1, \quad \sigma(2i-1) = 2i, \quad \forall i \in \mathbb{N}.
\]
It is easy to see that
\begin{align*}
    \sum_{i=1}^{\infty}A(i, \sigma(i))&=\sum_{i=1}^{\infty}A(2i,2i-1)+\sum_{i=1}^{\infty}A(2i-1,2i)\\
    &=2\sum_{i=1}^{\infty}\frac{1}{2^{\min\{2i,2i-1\}}}\\
    &=2 \sum_{i=1}^{\infty}\frac{1}{2^{2i-1}}=\frac{4}{3}\;,
\end{align*}
where the second equality uses the symmetry of $A$. One can easily construct a sequence $\tau_n\in S_{\N}$ such that 
\[\sum_{i=1}^{\infty}A(i, \tau_n(i)) \to \sum_{i=1}^{\infty}A(i, \sigma(i))\;,\]
as $n\to \infty$. In particular, 
\[\sup_{\sigma\in S_{\mathbb{N}}} \sum_{i=1}^{\infty}A(i, \sigma(i))\geq \frac{4}{3}>\|A\|_{\ell^2}^2\;.\]

\end{example}

\subsection{A continuous analogue of Marcus--Ree inequality}
\label{sec:MR_kernel}
Let $\mu$ be a measure on $[0, 1]$. A measure $\pi$ on $[0, 1]^2$ is said to be a \emph{coupling} of $\pi$ if
\begin{equation}
\label{eqn:couplingDefinition}
    \pi(A\times [0, 1])=\pi([0, 1]\times A)=\mu(A)\;\;.
\end{equation}
When $\mu$ is the uniform measure on the set $\{i/n: i\in [n]\}\subseteq [0, 1]$, any coupling of $\pi$ of $\mu$ can be written as an $n\times n$ bistochastic matrix $A$ where $A(i, j)=\pi((i, j))$. In this light, Marcus--Ree inequality is essentially a result about the couplings of uniform measure on discrete spaces. Also, recall that the Marcus--Ree inequality~\eqref{eqn:MarcusRee} follows from a more general inequality namely 
\begin{equation}
\label{eqn:Marcus-Ree-General}
    \sup_{B\in \Omega_n}\langle A, B\rangle_{F}\leq \sup_{P\in \mathfrak{S}_{n}}\langle A, P\rangle_{F}\;,
\end{equation}
where $\mathfrak{S}_n$ denotes the set of all $n\times n$ permutation matrices. Birkhoff's theorem identifies the extreme points of $\Omega_n$ as the set of permutation matrices. The proof of the above inequality follows from Birkhoff's theorem and observing that $B\mapsto \langle A, B\rangle_{F}$ is a convex function on $\Omega_n$.

In this section, we extend the Marcus--Ree inequality for kernels that can be seen as a coupling of uniform measure on $[0, 1]$. A \emph{kernel} is a bounded measurable function $W:[0, 1]^2\to \R$. We say that a kernel $W$ is \emph{bistochastic} if $W(x, y)\geq 0$ for a.e. $(x, y)\in [0, 1]^2$ and 
\[\int W(x, z)\diff z = 1=\int W(z, x)\diff z\;,\]
for a.e. $x\in [0, 1]$. We denote the set of all bistochastic kernels by $\Wcal_{\mathrm{bistoch}}$.
One can think of a bistochastic kernel $W$ as a copula, i.e. a coupling of uniform measure on $[0, 1]$. Let $W$ be a bistochastic kernel. Let $\pi_{W}$ be measure on $[0, 1]^2$ with density $W$ with respect to the uniform measure on $[0, 1]^2$, that is, 
\[\pi(\diff x\diff y) = W(x, y)\diff x\diff y\;.\]
It is easy to verify that $\pi_{W}$ is a coupling of Lebesgue measure on $[0, 1]$. We denote by $\Pi$ the set of all couplings of Lebesgue measure on $[0, 1]$. We now state a Marcus--Ree inequality for bistochastic kernels. Let $\Tcal$ denote the set of all Lebesgue measure-preserving maps $T:[0, 1]\to [0, 1]$. For brevity, we refer to such maps as measure-preserving maps. 
\begin{theorem}
\label{thm:continum}
    Let $W$ be a continuous bistochastic kernel. Then, 
    \[\sup_{\pi\in \Pi}\int W(x, y)\pi(\diff x\diff y) \leq \sup_{T\in \Tcal}\int W(x, T(x))\diff x \;. \]
    Taking $\pi=\pi_{W}$ we obtain 
    \[ \norm{2}{W}^2\leq \sup_{T\in \Tcal}\int W(x, T(x))\diff x\;.\]
\end{theorem}

We point out that the analog of Birkhoff's theorem for $\Wcal_{\mathrm{bistoch}}$ is false. The extreme points of $\Pi$ are not necessarily concentrated on the graph of a function. In fact, there are extreme points in $\Pi$ whose support is the full square $[0, 1]^2$~\cite{losert1982counter}. Therefore, the standard proof of Marcus--Ree inequality does not directly extend to this case. This is what makes Theorem~\ref{thm:continum} interesting.

We begin with some preparation before we start the proof. For $n\in \N$, let $\Dcal_n$ denote the dyadic partition of $[0, 1]$. That is, $\Dcal_n=\{I^n_{k}: 0\leq k\leq 2^n-1\}$ where 
$I^n_{k}=[k2^{-n}, (k+1)2^{-n})$ for all $1\leq k\leq 2^n-2$ and
$I^n_{2^n-1}=[1-2^{-n}, 1]$. Let $\Fcal_n$ denote the $\sigma$-algebra generated by the sets of the form $I^n_{i}\times I^n_j$ where $I^n_i, I^n_j\in \Dcal_n$. We say that a kernel $W$ is finite-dimensional if it is measurable with respect to the sigma-algebra $\Fcal_n$ for some $n$. Given a kernel $W$, we define a finite-dimensional kernel $W_n=\mathbb{E}(W\vert \Fcal_n)$. Note that $W_n$ is constant on $I^n_{i}\times I^n_j$ and 
\[W_n(x, y) = \frac{1}{|I^n_i\times I^n_j|}\int_{I^n_i\times I^n_j} W(u, v)\diff u\diff v, \quad \text{if }\; (x, y)\in I^n_i\times I^n_j\;.\]
We make some observations.
\begin{enumerate}
    \item If $W$ is bistochastic, then so is $W_n$. 
    \item By Martingale convergence theorem, $W_n\to W$ a.e. in $L^2$ as $n\to \infty$. If $W$ is continuous, then $W_n\to W$ in $L^{\infty}$.
\end{enumerate}
Let $W_n\in \Fcal_n$ be a finite-dimensional symmetric kernel. We associate to $W_n$ a $2^n\times 2^n$ bistochastic matrix $A_n$ as
\[
A_n(i,j) \coloneqq 2^{-n} W_n(x, y), \quad \text{if}\;\; (x, y)\in I^n_{i}\times I^n_j\;.
\]
Notice that $\| A_n \|_{F}^2 = \| W_n \|_{L^2}^2$. Let $\sigma\in S_{2^n}$ be a permutation. Let $T_\sigma:[0, 1]\to [0, 1]$ be the Lebesgue measure-preserving map which is an affine linear homeomorphism taking $I^n_{k}$ to $I^n_{\sigma(k)}$ for each $k$. Then,
\[
\sum_{i=1}^{n} A_n(i, \sigma(i)) = \int W(x, T_{\sigma}(x))\diff x\;.
\]
In particular, we get the following lemma as an immediate consequence of the Marcus--Ree inequality.
\begin{lemma}
\label{lemma:FiniteDimensional}
Let $W_n\in \Fcal_n$ be a finite-dimensional bistochastic kernel. Then, 
\[
\| W_n \|_{L^2}^2 = \| A_n \|_{F}^2 \leq \max_{\sigma\in S_{2^n}} \sum_{i=1}^{n} A_n(i, \sigma(i)) = \sup_{\sigma\in S_{2^n}}\int W_n(x, T_{\sigma}(x))\diff x\;.
\]
\end{lemma}

\begin{proof}[Proof of Theorem~\ref{thm:continum}]
Let $W\in \Wcal_{\mathrm{bistoch}}$ and let $\pi\in \Pi$. For $n\in \N$, let $W_n=\mathbb{E}(W\vert \Fcal_n)$. Let $A_n$ be the bistochastic matrix corresponding to $W_n$ and let $B_n$ be the $2^n\times 2^n$ bistochastic matrix such that 
\[B_n(i, j) = \pi(I_{i}^n\times I_j^{n})\;.\]
 Observe that 
\begin{align*}
    \int \int W_n(x, y)\pi(\diff x\;\diff y) = \langle A_n, B_n\rangle\;.
\end{align*}
Using Marcus--Ree inequality and Lemma~\ref{lemma:FiniteDimensional} we obtain
\begin{equation}
\label{eqn:Aux1}
    \int \int W_n(x, y)V(y, x)\diff y\diff x \leq \sup_{\sigma\in S_{2^n}}\int W_n(x, T_{\sigma}(x))\diff x\;.
\end{equation}
Now observe that 
\begin{align*}
     \int \int W(x, y)\pi(\diff y\diff x) &\leq   \int \int W_n(x, y)\pi(\diff y\diff x )+ \norm{\infty}{W_n-W}\\
     &\leq \sup_{\sigma\in S_{2^n}}\int W_n(x, T_{\sigma}(x))\diff x +\norm{\infty}{W_n-W}\;.
\end{align*}
Further notice that for any measure preserving map $T$ we have
\[\int W_n(x, T(x))\diff x \leq \int W(x, T(x))\diff x + \norm{\infty}{W-W_n}\leq \sup_{T\in \Tcal} W(x, T(x))\diff x+\norm{\infty}{W-W_n} \;.\]
Combining all this, we obtain 
\[   \int \int W(x, y)V(y, x)\diff y\diff x \leq \sup_{T\in \Tcal} W(x, T(x))\diff x+ 2\norm{\infty}{W_n-W}\;.\]
The proof finishes by letting $n\to \infty$ and noting that $\norm{\infty}{W_n-W}\to 0$ as $n\to \infty$ by the Martingale convergence theorem and taking supremum over $\pi\in \Pi$. 

\end{proof}

\subsection{An economic interpretation of Marcus--Ree inequality}
\label{sec:EconomicInterpretation}
Consider a scenario in which $n$ resources/agents are to be assigned to $n$ jobs. Suppose that the value or the payoff of assigning the resource/agent $i$ to job $j$ is given by $A(i, j)$. The problem of finding an optimal assignment or pairing of resources/agents and jobs that maximizes the total value is a typical problem studied in combinatorial optimization. Clearly, the maximal trace of the matrix $A$, $\mathrm{maxtrace}(A)=\max_{\sigma\in S_n}\sum_{i=1}^{n}A(i, \sigma(i))$ gives the maximal possible value or the value of optimal assignment. We now present an economic interpretation of the Marcus--Ree inequality. For this, it is better to work with the more general version of Marcus--Ree inequality, given in~\eqref{eqn:Marcus-Ree-General}. Suppose that each resource can be distributed among the jobs arbitrarily. We further assume that the distribution of the resources is costless. Then, instead of assigning the resource $i$ to the job $j$, we can assign $B(i, j)$ fraction of resource $i$ to job $j$ for each $j$. This yields a bistochastic matrix $B$ that we will refer to as the distribution plan. With this interpretation, $\langle A, B\rangle$ is the value of the assignment plan $B$. It may seem that using a distribution plan gives greater flexibility in resource assignment. However, the Marcus--Ree inequality~\eqref{eqn:Marcus-Ree-General} says that the flexibility to divide resources arbitrarily can not beat the optimal assignment where each resource is fully devoted to a unique job.

\bibliographystyle{alpha} 
\bibliography{ref}
\newpage
\appendix
\section{\texorpdfstring{$4 \times 4$}{4 x 4} Erd\H{o}s matrices}
\label{append:Erdos4x4}

\begin{center}
\begin{longtable}{c c c c}
$\begin{pmatrix} 1 & 0 & 0 & 0 \\ 0 & 1 & 0 & 0 \\ 0 & 0 & 1 & 0 \\ 0 & 0 & 0 & 1 \end{pmatrix}, $ &
$\dfrac{1}{2} \begin{pmatrix} 1 & 0 & 0 & 1 \\ 0 & 2 & 0 & 0 \\ 0 & 0 & 2 & 0 \\ 1 & 0 & 0 & 1 \end{pmatrix}, $ &
$\dfrac{1}{2} \begin{pmatrix} 1 & 0 & 0 & 1 \\ 0 & 2 & 0 & 0 \\ 1 & 0 & 1 & 0 \\ 0 & 0 & 1 & 1 \end{pmatrix}, $ &
$\dfrac{1}{5} \begin{pmatrix} 1 & 0 & 2 & 2 \\ 0 & 5 & 0 & 0 \\ 2 & 0 & 3 & 0 \\ 2 & 0 & 0 & 3 \end{pmatrix}, $\\[2.7em]
$\dfrac{1}{4} \begin{pmatrix} 1 & 0 & 1 & 2 \\ 0 & 4 & 0 & 0 \\ 2 & 0 & 2 & 0 \\ 1 & 0 & 1 & 2 \end{pmatrix},$ &
$\dfrac{1}{2} \begin{pmatrix} 1 & 0 & 0 & 1 \\ 0 & 1 & 1 & 0 \\ 1 & 0 & 1 & 0 \\ 0 & 1 & 0 & 1 \end{pmatrix}, $ &
$ \dfrac{1}{2} \begin{pmatrix} 1 & 0 & 0 & 1 \\ 0 & 1 & 1 & 0 \\ 0 & 1 & 1 & 0 \\ 1 & 0 & 0 & 1 \end{pmatrix}, $ &
$ \dfrac{1}{3} \begin{pmatrix} 1 & 0 & 1 & 1 \\ 0 & 3 & 0 & 0 \\ 1 & 0 & 1 & 1 \\ 1 & 0 & 1 & 1 \end{pmatrix}, $\\[2.7em]
$ \dfrac{1}{3} \begin{pmatrix} 1 & 0 & 0 & 2 \\ 1 & 2 & 0 & 0 \\ 1 & 0 & 2 & 0 \\ 0 & 1 & 1 & 1 \end{pmatrix}, $ & $ \dfrac{1}{23} \begin{pmatrix} 4 & 0 & 9 & 10 \\ 10 & 13 & 0 & 0 \\ 9 & 0 & 14 & 0 \\ 0 & 10 & 0 & 13 \end{pmatrix}, $ &
$ \dfrac{1}{8} \begin{pmatrix} 1 & 1 & 3 & 3 \\ 4 & 4 & 0 & 0 \\ 0 & 3 & 5 & 0 \\ 3 & 0 & 0 & 5 \end{pmatrix}, $ & 
$ \dfrac{1}{8} \begin{pmatrix} 1 & 0 & 3 & 4 \\ 3 & 5 & 0 & 0 \\ 3 & 0 & 5 & 0 \\ 1 & 3 & 0 & 4 \end{pmatrix}, $\\[2.7em]

$ \dfrac{1}{11} \begin{pmatrix} 2 & 0 & 3 & 6 \\ 4 & 7 & 0 & 0 \\ 5 & 0 & 6 & 0 \\ 0 & 4 & 2 & 5 \end{pmatrix}, $ &
$ \dfrac{1}{43} \begin{pmatrix} 2 & 7 & 15 & 19 \\ 19 & 24 & 0 & 0 \\ 15 & 0 & 28 & 0 \\ 7 & 12 & 0 & 24 \end{pmatrix}, $ &
$ \dfrac{1}{4} \begin{pmatrix} 1 & 0 & 1 & 2 \\ 2 & 2 & 0 & 0 \\ 0 & 2 & 2 & 0 \\ 1 & 0 & 1 & 2 \end{pmatrix}, $ &
$ \dfrac{1}{4} \begin{pmatrix} 1 & 0 & 1 & 2 \\ 1 & 2 & 1 & 0 \\ 2 & 0 & 2 & 0 \\ 0 & 2 & 0 & 2 \end{pmatrix}, $ \\[2.7em]
$ \dfrac{1}{29} \begin{pmatrix} 3 & 6 & 6 & 14 \\ 13 & 16 & 0 & 0 \\ 13 & 0 & 16 & 0 \\ 0 & 7 & 7 & 15 \end{pmatrix}, $ &
$ \dfrac{1}{29} \begin{pmatrix} 3 & 0 & 13 & 13 \\ 14 & 15 & 0 & 0 \\ 6 & 7 & 16 & 0 \\ 6 & 7 & 0 & 16 \end{pmatrix}, $ &
$ \dfrac{1}{7} \begin{pmatrix} 2 & 0 & 2 & 3 \\ 0 & 4 & 3 & 0 \\ 2 & 3 & 2 & 0 \\ 3 & 0 & 0 & 4 \end{pmatrix}, $ & 
$ \dfrac{1}{14} \begin{pmatrix} 1 & 3 & 3 & 7 \\ 6 & 8 & 0 & 0 \\ 6 & 0 & 8 & 0 \\ 1 & 3 & 3 & 7 \end{pmatrix}, $ \\[2.7em]
$ \dfrac{1}{14} \begin{pmatrix} 1 & 1 & 6 & 6 \\ 7 & 7 & 0 & 0 \\ 3 & 3 & 8 & 0 \\ 3 & 3 & 0 & 8 \end{pmatrix}, $ &
$ \dfrac{1}{22} \begin{pmatrix} 4 & 4 & 5 & 9 \\11 & 11 & 0 & 0 \\ 7 & 7 & 8 & 0 \\ 0 & 0 & 9 & 13 \end{pmatrix}, $ &
$ \dfrac{1}{22} \begin{pmatrix} 4 & 0 & 7 & 11 \\ 9 & 13 & 0 & 0 \\ 5 & 9 & 8 & 0 \\ 4 & 0 & 7 & 11 \end{pmatrix}, $ & 
$ \dfrac{1}{10} \begin{pmatrix} 2 & 0 & 3 & 5 \\ 5 & 5 & 0 & 0 \\ 3 & 3 & 4 & 0 \\ 0 & 2 & 3 & 5 \end{pmatrix}, $\\[2.7em]
$ \dfrac{1}{19} \begin{pmatrix} 2 & 3 & 5 & 9 \\ 9 & 10 & 0 & 0 \\ 5 & 6 & 8 & 0 \\ 3 & 0 & 6 & 10 \end{pmatrix}, $ &
$ \dfrac{1}{18} \begin{pmatrix} 2 & 2 & 5 & 9 \\ 9 & 9 & 0 & 0 \\ 5 & 5 & 8 & 0 \\ 2 & 2 & 5 & 9 \end{pmatrix}, $ & 
$ \dfrac{1}{17} \begin{pmatrix} 5 & 0 & 5 & 7 \\ 6 & 5 & 6 & 0 \\ 6 & 5 & 6 & 0 \\ 0 & 7 & 0 & 10 \end{pmatrix}, $ & 
$ \dfrac{1}{4} \begin{pmatrix} 1 & 1 & 1 & 1 \\ 0 & 2 & 0 & 2 \\ 2 & 0 & 2 & 0 \\ 1 & 1 & 1 & 1 \end{pmatrix}, $\\[2.7em]
$ \dfrac{1}{4} \begin{pmatrix} 1 & 0 & 1 & 2 \\ 1 & 2 & 1 & 0 \\ 1 & 2 & 1 & 0 \\ 1 & 0 & 1 & 2 \end{pmatrix}, $ & 
$ \dfrac{1}{8} \begin{pmatrix} 1 & 2 & 2 & 3 \\ 2 & 3 & 3 & 0 \\ 2 & 3 & 3 & 0 \\ 3 & 0 & 0 & 5 \end{pmatrix}, $ & $ \dfrac{1}{14} \begin{pmatrix} 3 & 3 & 4 & 4 \\ 7 & 7 & 0 & 0 \\ 0 & 4 & 5 & 5 \\ 4 & 0 & 5 & 5 \end{pmatrix}, $ & 
$ \dfrac{1}{14} \begin{pmatrix} 3 & 0 & 4 & 7 \\ 4 & 5 & 5 & 0 \\ 4 & 5 & 5 & 0 \\ 3 & 4 & 0 & 7 \end{pmatrix}, $ \\[2.7em]
$ \dfrac{1}{13} \begin{pmatrix} 2 & 3 & 4 & 4 \\ 6 & 7 & 0 & 0 \\ 2 & 3 & 4 & 4 \\ 3 & 0 & 5 & 5 \end{pmatrix}, $ &
$ \dfrac{1}{13}
\begin{pmatrix}
2 & 2 & 3 & 6 \\
4 & 4 & 5 & 0 \\
4 & 4 & 5 & 0 \\
3 & 3 & 0 & 7
\end{pmatrix}, $ &
$ \dfrac{1}{3}
\begin{pmatrix}
1 & 0 & 1 & 1 \\
0 & 1 & 1 & 1 \\
1 & 1 & 1 & 0 \\
1 & 1 & 0 & 1
\end{pmatrix}, $ &
$ \dfrac{1}{6}
\begin{pmatrix}
1 & 1 & 2 & 2 \\
3 & 3 & 0 & 0 \\
1 & 1 & 2 & 2 \\
1 & 1 & 2 & 2
\end{pmatrix}, $ \\[2.7em]
$ \dfrac{1}{6}
\begin{pmatrix}
1 & 1 & 1 & 3 \\
2 & 2 & 2 & 0 \\
2 & 2 & 2 & 0 \\
1 & 1 & 1 & 3
\end{pmatrix}, $ &
$ \dfrac{1}{11}
\begin{pmatrix}
2 & 3 & 3 & 3 \\
3 & 4 & 0 & 4 \\
3 & 4 & 4 & 0 \\
3 & 0 & 4 & 4
\end{pmatrix}, $ &
$ \dfrac{1}{10}
\begin{pmatrix}
2 & 2 & 3 & 3 \\
2 & 2 & 3 & 3 \\
3 & 3 & 4 & 0 \\
3 & 3 & 0 & 4
\end{pmatrix}, $ &
$ \dfrac{1}{9}
\begin{pmatrix}
2 & 2 & 2 & 3 \\
2 & 2 & 2 & 3 \\
3 & 3 & 3 & 0 \\
2 & 2 & 2 & 3
\end{pmatrix}, $\\[2.7em]
$ \dfrac{1}{4}
\begin{pmatrix}
1 & 1 & 1 & 1 \\
1 & 1 & 1 & 1 \\
1 & 1 & 1 & 1 \\
1 & 1 & 1 & 1
\end{pmatrix}$. &&&
\end{longtable}
\end{center}

\end{document}